\documentclass[12pt,leqno]{article}

\usepackage[lite]{amsrefs}
\usepackage{latexsym,enumerate}
\usepackage[notref, notcite]{showkeys}
\usepackage{amssymb, amsfonts, eucal,amsmath,amsthm}
\usepackage[cp1251]{inputenc}
\usepackage[OT2,T2A]{fontenc}
\usepackage[ukrainian,english]{babel}

\newtheorem{theorem}{Theorem}[section]

\newtheorem{lemma}[theorem]{Lemma}

\newtheorem{remark}[theorem]{Remark}

\newcommand{\sect}[1]{\section{#1} \setcounter{equation}{0} }

\newcounter{ca}
 \newcommand{\ec}{\end{comment}}

\def\be  {\begin{equation}}
\def\ee  {\end{equation}}
\def\ba  {\begin{eqnarray}}
\def\ea  {\end{eqnarray}}
\def\baa {\begin{eqnarray*}}
\def\eaa {\end{eqnarray*}}
\newenvironment{comment}[2]
{\bgroup\vspace{7pt}
\begin{tabular}{|p{5in}|}
\hline \qquad \bf \footnotesize Comment -- to be deleted in the final version \\
\hline
\quad\sl\footnotesize #1#2} {\\ \hline \end{tabular}
\vspace{7pt}\indent\egroup}

\def\updots{\mathinner{\mkern
1mu\raise 1pt \hbox{.}\mkern 2mu \mkern 2mu \raise
4pt\hbox{.}\mkern 1mu \raise 7pt\vbox {\kern 7 pt\hbox{.}}} }

\def \const{\mathop{\rm const}\nolimits}

\title{{\sc Shape preserving approximation of periodic functions -- Conclusion}\thanks{{\it AMS classification:} 41A29, 42A10, 41A25, {\it Keywords
and phrases:} Comonotone approximation by trigonometric polynomials, Degree of approximation, Jackson-type estimates.}}
\author{D. Leviatan\thanks{Raymond and Beverly Sackler School of Mathematical Sciences, Tel Aviv University, Tel Aviv 6139001, Israel
({\tt  leviatan@tauex.tau.ac.il}).}
\and
I. O. Shevchuk\thanks
{Faculty of Mechanics and Mathematics, Taras
Shevchenko National University of Kyiv, 01601 Kyiv, Ukraine ({\tt shevchukio@knu.ua}).}
}

\begin{document}

\maketitle

\begin{abstract}
We give here the final results about the validity of Jackson-type estimates in comonotone approximation of $2\pi$-periodic functions by trigonometric polynomials. For coconvex and the so called co-$q$-monotone, $q>2$, approximations, everything is known by now. Thus, this paper concludes the research on Jackson type estimates of Shape Preserving Approximation of periodic functions by trigonometric polynomials. It is interesting to point out that the results for comonotone approximation of a periodic function are substantially different than the analogous results for comonotone approximation, by algebraic polynomials, of a continuous function on a finite interval.

\end{abstract}

\sect{Introduction}
We give here the final results about the validity of Jackson-type estimates in comonotone approximation of $2\pi$-periodic functions by trigonometric polynomials. For coconvex and the so called co-$q$-monotone, $q>2$, approximations, everything is known by now (see, e.g, \cite{LS}). Thus, this paper concludes the research on Jackson type estimates of Shape Preserving Approximation of periodic functions by trigonometric polynomials.

Let $C=C^0$ be the space of $2\pi$-periodic continuous functions $f$ on $\mathbb R$, endowed with the uniform norm
$$
\|f\|:=\|f\|_{\mathbb R}=\max_{x\in\mathbb R}|f(x)|.
$$
For $r\in\mathbb N$, denote by $C^r$ the subspace of $r$ times continuously differentiable functions $f\in C$.

In this paper we discuss the approximation of $f\in C^r$, $r\ge0$, that changes its monotonicity a few times in its period $(-\pi,\pi]$, by trigonometric polynomials that follow the monotonicity of $f$ (i.e., comonotone approximation). Note that a nonconstant trigonometric polynomial may only have a finite even number of changes of monotonicity in a period. Hence, one is limited to discussion of comonotone approximation of $f$ that has  $2s$, $s\ge1$, changes of monotonicity.

Let $Y_s:=\{y_i\}_{i\in\mathbb Z}$, $s\in\mathbb N$, $y_i\in\mathbb R$, be such that $y_i<y_{i+1}$ and $y_{i+2s}=y_i+2\pi$, $i\in\mathbb Z$. Denote by $\Delta^{(1)}(Y_s)$,  the collection of functions $f\in C$, such that $(-1)^{i-1}f$ is nondecreasing on the intervals, $[y_{i-1},y_i]$, and note that
if we denote
\begin{equation}\label{Pi}
\Pi(x):=\prod_{i=1}^{2s}\sin\frac{x-y_i}2,
\end{equation}
then for $f\in C^1$, $f\in\Delta^{(1)}(Y_s)$, if and only if
$$
f'(x)\Pi(x)\ge0,\quad x\in\mathbb R.
$$
For $f\in\Delta^{(1)}(Y_s)$, let
$$
E_n^{(1)}(f,Y_s):=\inf_{T_n\in\Delta^{(1)}(Y_s)}\|f-T_n\|,
$$
denote the error of best approximation of $f$ by comonotone trigonometric polynomials $T_n$ of degree $<n$.

We are interested in estimates of the Jackson type,
\begin{equation}\label{m}
E_n^{(1)}(f,Y_s)\le\frac{c(k,r,s)}{n^r}\omega_k\left(f^{(r)},\frac1n\right),\quad n\ge 1,
\end{equation}
where $\omega_k\left(f^{(r)},\cdot\right)$ is the $k$th modulus of smoothness of $f^{(r)}$.
Specifically, we are interested in the validity of (\ref{m}), depending on the various parameters, $s\ge1$, $k\ge1$ and $r\ge0$.

The first result in this direction was obtained by Lorentz and Zeller \cite{LZ}*{Theorem 1}, who proved such estimate (with $r=0$ and $k=1$) for the approximation of a continuous bell-shaped function on $[-\pi,\pi]$, by
bell-shaped trigonometric polynomials. Clearly, such a bell-shaped function may be looked upon as an even $2\pi$-periodic function, decreasing on $[-\pi,0]$ and increasing on $[0,\pi]$. (They applied this result to get a corresponding estimate for the approximation of a continuous monotone function on $[-1,1]$, by monotone algebraic polynomials.)

If $f\in C^r\cap\Delta^{(1)}(Y_s)$, and either $k=1$ or $r>2s$, then the validity of (\ref{m}) is proved in \cite{LSS}*{Theorem 1.1}.
(For $r=0$ and $k=1$, this was shown earlier by \cite{Pl}.)

In this paper we determine all triplets $(r,k,s)$, for which (\ref{m}) is valid, and prove that (\ref{m}) is invalid for all other triplets.

\begin{remark} If $(\ref{m})$ is invalid and $r\ge2$, then all is not lost, as it still follows by \cite{Dz09}, that the estimate
\begin{equation}\label{lm}
E_n^{(1)}(f,Y_s)\le \frac {c(k,r,s)}{n^r}\omega_k\left(f^{(r)},\frac1n\right),\quad n\ge N(r,k,Y_s),
\end{equation}
holds, where $N(r,k,Y_s)$ depends on the smallest distance between the points $y_i$, $i\in\mathbb Z$. In addition, for $n>N(Y_s)$ we may guarantee that the distance between any two consecutive $y_i$'s is bigger than $\frac{5\pi}n$. Therefore, if we take $n>N(Y_s)$, then we may combine Lemma $\ref{4.6}$ and Lemma $\ref{mon}$ with Lemma $\ref{lss}$ below, and get $(\ref{lm})$ for $r=1$ and $k=2,3$, as well. Hence, $(\ref{lm})$ is valid, but as we show in this paper, $(\ref{m})$ is valid for these pairs only for $s=1$, and is invalid otherwise. Finally, the validity of $(\ref{lm})$ when $r=0$ and $k=2$, is proved in \cite{DP}. However, if $r=0$ and $k\ge3$, and if $r=1$ and $k\ge4$, then even $(\ref{lm})$ is invalid, since (see \cite{DVY}), for each $Y_s$ there is a function $f\in C^r\cap\Delta^{(1)}(Y_s)$, such that for these pairs $(r,k)$,
$$
\limsup_{n\to\infty}\frac{n^r}{\omega_k(f^{(r)},1/n)}E_n^{(1)}(f,Y_s)=\infty.
$$
\end{remark}
Recall that for $g\in C$ or $g\in C[a,b]$, we denote the $k$th modulus of smoothness of $g$, $k\ge1$, respectively, by
\[
\omega_k(g,t)=\sup_{0<h\le t}\left\|\Delta_h^k(g,\cdot)\right\|\quad\text{and}\quad\omega_k(g,t;[a,b])=\sup_{0<h\le t}\left\|\Delta_h^k(g,\cdot)\right\|_{[a,b-kh]},
\]
where $\Delta_h^k(g,\cdot)=\sum_{i=0}^k(-1)^i\binom kig\bigl(\cdot+ih\bigr)$ and $\|g\|_{[c,d]}=\max_{x\in[c,d]}|g(x)|$.
\sect{Main results}
Our first positive result is
\begin{theorem}\label{T1} Let $r=2s-2$. If $f\in C^r\cap\Delta^{(1)}(Y_s)$, then $(\ref{m})$ with $k=2$ holds.
\end{theorem}
However, Theorem \ref{T1} cannot be extended to higher order moduli. Namely, we prove
\begin{theorem}\label{L53} Let $r=2s-2$. There is a constant $c(r)>0$, such that for each $n\ge1$, there exist a collection $Y_s$ and a function $F\in C^r\cap\Delta^{(1)}{(Y_s)}$, such that
\begin{equation}\label{l510}
E_n^{(1)}(F,Y_s)>\sqrt n\frac{c(r)}{n^r}\omega_3\left(F^{(r)},\frac1n\right).
\end{equation}
Hence, for every $A>0$ and $N>0$, there exist $n>N$, $Y_s$ and $F\in C^r\cap\Delta^{(1)}{(Y_s)}$, such that
\[
E_n^{(1)}(F,Y_s)>\frac A{n^r}\omega_3\left(F^{(r)},\frac1n\right).
\]
\end{theorem}
We also prove
\begin{theorem}\label{T2} Let $r=2s-1$. If $f\in C^r\cap\Delta^{(1)}(Y_s)$, then
\begin{equation}\label{150}
E_n^{(1)}(f,Y_s)\le\frac{c(r)}{n^r}\omega_3\left(f^{(r)},\frac1n\right),\quad n\ge1.
\end{equation}
\end{theorem}
\begin{remark} Theorem $\ref{T2}$ clearly implies $(\ref{150})$ with $\omega_2$ and $\omega_1$ replacing $\omega_3$. The case of $\omega_1$ was proved in \cite{LSS}*{Theorem 1.1}, while that of $\omega_2$ is new.
\end{remark}
Again, Theorem \ref{T2} cannot be extended to higher order moduli. Namely, we prove
\begin{theorem}\label{L52} Let $r=2s-1$. There is a constant $c(r)>0$, such that for each $n\ge1$, there exist a collection $Y_s$ and a function $F\in C^r\cap\Delta^{(1)}{(Y_s)}$, such that
\begin{equation}\label{l51}
E_n^{(1)}(F,Y_s)>\root3\of n\frac{c(r)}{n^r}\omega_4\left(F^{(r)},\frac1n\right).
\end{equation}
Hence, for every $A>0$ and $N>0$, there exist $n>N$, $Y_s$ and $F\in C^r\cap\Delta^{(1)}{(Y_s)}$, such that
\[
E_n^{(1)}(F,Y_s)>\frac A{n^r}\omega_4\left(F^{(r)},\frac1n\right).
\]
\end{theorem}
Our last positive result is
\begin{theorem}\label{T3} Let $r=2s$ and let $k\ge3$. If $f\in C^r\cap\Delta^{(1)}(Y_s)$, then
\begin{equation}\label{151}
E_n^{(1)}(f,Y_s)\le\frac{c(r,k)}{n^r}\omega_k\left(f^{(r)},\frac1n\right),\quad n\ge1.
\end{equation}
\end{theorem}
\begin{remark} Theorem $\ref{T3}$, evidently implies $(\ref{151})$ also for $k=1,2$, which were proved in \cite{LSS}*{Theorem 1.1}.
\end{remark}
Finally, if $0\le r<2s-2$ and $k\ge2$, then (\ref{m}) cannot be had. Namely,
\begin{theorem}\label{L51} Let $0\le r<2s-2$. There is a constant $c(r)>0$, such that for each $n\ge1$, there exist a collection $Y_s$ and a function $F\in C^r\cap\Delta^{(1)}{(Y_s)}$, such that
\begin{equation}\label{l52}
E_n^{(1)}(F,Y_s)> n\frac{c(r)}{n^r}\omega_2\left(F^{(r)},\frac1n\right).
\end{equation}
Hence, for every $A>0$ and $N>0$, there exist $n>N$, $Y_s$ and an $F\in C^r\cap\Delta^{(1)}{(Y_s)}$, such that
\[
E_n^{(1)}(F,Y_s)>\frac A{n^r}\omega_2\left(F^{(r)},\frac1n\right).
\]
\end{theorem}
We can now summarize the results. Given $s\ge1$, the pairs $(r,k)$, for  which (\ref{m}) holds, are
$k=1$ and $r\ge0$, $k=2$ and $r\ge 2s-2$, $k=3$ and $r\ge 2s-1$, and $k\ge 4$ and $r\ge 2s$.  For all other pairs (\ref{m}) is invalid. In the tables below we denote by ``$+"$ the pairs $(r,k)$ for which (\ref{m}) is valid.
We denote by ``$\oplus"$ the pairs $(r,k)$ for which (\ref{m}) is invalid but (\ref{lm}) is valid, and finally, ``$-"$ denotes that even (\ref{lm}) is invalid.
$$
\begin{matrix}
r & \vdots & \vdots & \vdots & \vdots & \vdots & \updots \\
2s+1 & + & + & + & + & + & \cdots \\
2s & + & + & + & + & + & \cdots \\
2s-1 & + & + & + & \oplus & \oplus & \cdots \\
2s-2 & +  & + & \oplus & \oplus & \oplus & \cdots \\
2s-3 & + & \oplus & \oplus & \oplus & \oplus & \cdots \\
\vdots & \vdots & \vdots & \vdots & \vdots & \vdots & \vdots \\
2 & + & \oplus & \oplus & \oplus & \oplus & \cdots \\
1 & + & \oplus & \oplus & - & - & \cdots \\
0 & + & \oplus & - & - & - & \cdots \\
 & 1 & 2& 3 & 4 & 5 & k
\end{matrix}
\\
\qquad\qquad
\begin{matrix}
r & \vdots & \vdots & \vdots & \vdots & \vdots & \updots \\
5 & + & + & + & + & + & \cdots \\
4 & + & + & + & + & + & \cdots \\
3 & + & + & + & \oplus & \oplus & \cdots \\
2 & + & + & \oplus & \oplus & \oplus & \cdots \\
1 & + & \oplus & \oplus & - & - & \cdots \\
0 & + & \oplus & - & - & - & \cdots \\
 & 1 & 2& 3 & 4 & 5 & k
\end{matrix}
$$
$$
s>2\qquad\qquad\qquad\qquad\qquad\qquad s=2
$$
$$
\begin{matrix}
r & \vdots & \vdots & \vdots & \vdots & \vdots & \updots \\
3 & + & + & + & + & + & \cdots \\
2 & + & + & + & + & + & \cdots \\
1 & + & + & + & - & - & \cdots \\
0 & + & + & - & - & - & \cdots \\
 & 1 & 2& 3 & 4 & 5 & k
\end{matrix}
$$
$$
s=1
$$
We emphasize that these tables differ substantially from the corresponding tables for comonotone approximation of non-periodic functions on a finite interval (see, e.g, \cite{KLPS}).

\sect{Divided differences}
Let $m\in\mathbb N$ and let the function $g$ be defined at the distinct points $t_i\in\mathbb R$, $0\le i\le m$. Denote $[t_0;g]:=g(t_0)$ and by
\[
[t_0,t_1,\dots,t_m;g]:=\sum_{i=0}^m\frac{g(t_i)}{\prod_{j=0,\,j\ne i}^m(t_i-t_j)},
\]
the divided difference of order $m$ of $g$ at the  knots $t_i$.
\begin{lemma}\label{m1} Let $t_0<t_1<\cdots<t_m$. Assume that either
\begin{equation}\label{dba}
(-1)^{m-i}\bigl(g(t_i)-g(t_{i-1})\bigr)\ge0,\quad 1\le i\le m,
\end{equation}
or
\begin{equation}\label{dbb}
(-1)^{m-i}\bigl(g(t_i)-g(t_{i-1})\bigr)\le0,\quad 1\le i\le m.
\end{equation}
Then, respectively,
\begin{equation}\label{da}
[t_0,\dots,t_m;g]\ge0,\quad\text{or}\quad [t_0,\dots,t_m;g]\le0,
\end{equation}
and
\begin{equation}\label{dd}
(t_m-t_0)^m\bigl|[t_0,t_1,\dots,t_m;g]\bigr|\ge \max_{0\le i\le m}g(t_i)-\min_{0\le i\le m}g(t_i),
\end{equation}
Therefore, for $m\ge2$,
\begin{equation}\label{df}
\bigl|[t_0,\dots,t_{m};g]\bigr|=\frac{\bigl|[t_1,\dots,t_{m};g]\bigr|+\bigl|[t_0,\dots,t_{m-1};g]\bigr|}{t_m-t_0},
\end{equation}
and if $2\le r\le m$, then
\begin{equation}\label{dd2}
\bigl|[t_0,\dots,t_m;g]\bigr|\prod_{i=r}^m(t_i-t_0)
\ge\bigl|[t_1,\dots,t_r;g]\bigr|.
\end{equation}
\end{lemma}
\begin{proof} Since all statements depend only on $g(t_i)$, $0\le i\le m$, and there is an algebraic polynomial of degree $\le m$, interpolating $g$ at $t_i$, $1\le i\le m$, without loss of generality we assume that $g$ is that polynomial. Then (\ref{dba}), respectively, (\ref{dbb}) implies that $g'(x)=a\prod_{i=1}^{m-1}(x-x_i)$, where $a\ge0$, respectively, $a\le0$, and $x_i\in(t_0,t_m)$, $1\le i\le m-1$. Note that $g^{(m)}(x)\equiv(m-1)!a$, so that $[t_0,\dots,t_m;g]=\frac am$. Thus, (\ref{da}) holds.

By the inequality between the arithmetic and geometric means,
\begin{align*}
\int_{t_0}^{t_m}|g'(x)|\,dx&=|a|\int_{t_0}^{t_m}\prod_{i=1}^{m-1}|x-x_i|\,dx\le\frac{|a|}{m-1}\sum_{i=1}^{m-1}\int_{t_0}^{t_m}|x-x_i|^{m-1}\,dx\\
&=\frac{|a|}{m(m-1)}\sum_{i=1}^{m-1}\bigl((t_m-x_i)^m+(x_i-t_0)^m \bigr)\le\frac{|a|}m(t_m-t_0)^m.
\end{align*}
Hence,
\begin{align*}
\max_{0\le i\le m}g(t_i)-\min_{0\le i\le m}g(t_i)&\le\int_{t_0}^{t_m}|g'(x)|\,dx\le\frac{|a|}m(t_m-t_0)^m\\
&=\bigl|[t_0,t_1,\dots,t_m;g]\bigr|(t_m-t_0)^m
\end{align*}
Thus, (\ref{dd}) is proved. Finally, (\ref{da}) readily implies (\ref{df}), and (\ref{dd2}) follows from (\ref{df}) by induction on $m-r$.
\end{proof}
\begin{lemma}\label{dl} Let $0\le l<m$. If $f\in C^l[a,b]$  and $x_i\in[a,b]$, $0\le i\le m$, are distinct, then
\begin{equation}\label{ineq}
|[x_0,\dots,x_m;f]|
\le\frac c{l!}\omega_{m-l}(f^{(l)},b-a;[a,b])\sum_{j=l}^m\frac1{\prod_{i=l,i\ne j}^m|x_j-x_i|},
\end{equation}
where $c$ is an absolute constant.
\end{lemma}
\begin{proof}
For $l\ge1$, let $U_l:=\{u_i\}_{i=1}^l$ and denote $f_l(x;U_l):=f^{(l)}\bigl(x_0+u_1(x_1-x_0)+\cdots+u_l(x-x_{l-1})\bigr)$.

By virtue of \cite{DL}{Chapter 4 (7.12)} we have,
\begin{align*}
[x_0,\dots,x_l;f]&=\int_0^1\int_0^{u_1}\dots\int_0^{u_{l-1}}f_l(x_l;U_l)\,du_l\cdots du_1\\
&=\int_0^1\int_0^{u_1}\dots\int_0^{u_{l-1}}[x_l;f_l(\cdot;U_l)]\,du_l\cdots du_1.
\end{align*}
Assume by induction that for some $m>l$,
\begin{equation}\label{dl1}
[x_0,\dots,x_{m-1};f]=\int_0^1\int_0^{u_1}\dots\int_0^{u_{l-1}}[x_l,\dots,x_{m-1};f_l(\cdot;U_l)]\,du_l\cdots du_1,
\end{equation}
we will prove (\ref{dl1}) for $m+1$.

Indeed,
\begin{align*}
&[x_0,\dots,x_m;f]=\frac{[x_0,\dots,x_{m-1};f]-[x_0,\dots,x_{m-2},x_m;f]}{x_{m-1}-x_m}\\
&=\int_0^1\int_0^{u_1}\dots\int_0^{u_{l-1}}\frac{[x_l,\dots,x_{m-1};f_l(\cdot;U_l)]-[x_l,\dots,x_{m-2},x_m;f_l(\cdot;U_l)]}{x_{m-1}-x_m}\,du_l\cdots du_1\\
&=\int_0^1\int_0^{u_1}\dots\int_0^{u_{l-1}}[x_l,\dots,x_m;f_l(\cdot;U_l)]\,du_l\cdots du_1.
\end{align*}
Hence,
\begin{equation}\label{dl2}
\bigl|[x_0,\dots,x_m;f]\bigr|\le\frac{\|f^{(l)}\|_{[a,b]}}{l!}\sum_{j=l}^m\frac1{\prod_{i=l,i\ne j}^m|x_j-x_i|}.
\end{equation}
Evidently, (\ref{dl2}) holds also for $l=0$.

Finally, by Whitney's inequality there is an algebraic polynomial $P$ of degree $<m-l$, such that $\|f^{(l)}-P\|\le c\omega_{m-l}(f^{(l)},b-a;[a,b])$. Since for any polynomial $Q$ such that $Q^{(l)}=P$, $[x_0,\dots,x_m;f]=[x_0,\dots,x_m;f-Q]$, we conclude that (\ref{dl2}) yields (\ref{ineq}).
\end{proof}
\sect{Auxiliary lemmas}
In the sequel the various positive constants $c$ and $c_i$ may depend only on $r\ge0$, $s\ge1$ and $k\ge2$. Also, we assume that $f\ne\const$.
\begin{lemma}\label{4.1} Let $h>0$ and $y_{2s}-y_1\le(6s-2)h$. If $f\in C^{2s-1}\cap\Delta^{(1)}(Y_s)$, then
\begin{equation}\label{101}
\|f-f(y_{2s})\|\le ch^{2s-1}\omega_{k+1}(f^{(2s-1)},h).
\end{equation}
Hence, if $f\in C^{2s}\cap\Delta^{(1)}(Y_s)$, then
\begin{equation}\label{102}
\|f-f(y_{2s})\|\le ch^{2s}\omega_k(f^{(2s)},h).
\end{equation}
\end{lemma}
\begin{proof}
Let $g(t):=f(t)-f(y_{2s})$ and note that $f\in \Delta^{(1)}(Y_s)$ implies $\|g\|=\|g\|_{[y_1,y_{2s}]}$. Set $H:=2^k6sh$, and we may assume that $kH<2\pi$, for otherwise, by Whitney's inequality, the periodicity of $f$ yields (\ref{101}). Let
$$
t_0:=y_1-h, \quad t_i:=y_i,\quad 1\le i\le 2s,\quad\text{and}\quad t_{{2s}+i}=t_0+iH,\quad1\le i\le k,
$$
and set
$
x_{2s}:=t_0,\quad x_i:=t_{i+1},\quad0\le i\le {2s}-1,\text{ and}\quad x_{{2s}+i}=t_{{2s}+i},\quad1\le i\le k.
$
Then $x_j-x_i\ge h$, $i\ne j$, ${2s}-1\le i,j\le {2s}+k$, so that by Lemma \ref{dl} with $l={2s}-1$ and $m=2s+k$,
we have
\begin{equation}\label{103}
h^{k+1}\bigl|[t_0,\dots,t_{{2s}+k};g]\bigr|\le c\omega_{k+1}(f^{({2s}-1)},h).
\end{equation}
Set
$$
\tilde g(t):=\begin{cases} g(t),\quad & t\in[t_0,t_{2s}],\\
0,\quad & t\in [t_{2s},t_0+2\pi),
\end{cases}
$$
and note that $\tilde g$ satisfies either (\ref{dba}) or (\ref{dbb}), with $m={2s}+k$. By (\ref{dd}) and (\ref{dd2})
\begin{align*}
(t_{2s}-t_1)^{{2s}-1}\bigl|[t_0,\dots,t_{{2s}+k};\tilde g]\bigr|\prod_{i={2s}}^{k+{2s}}(t_i-t_0)&\ge(t_{2s}-t_1)^{{2s}-1}\bigl|[t_1,\dots,t_{{2s}};\tilde g]\bigr|\\
&\ge \max_{1\le i\le {2s}}\tilde g(t_i)-\min_{1\le i\le {2s}}\tilde g(t_i)\\
&\ge\|\tilde g\|_{[t_1,t_{2s}]}=\|g\|_{[t_1,t_{2s}]}=\|g\|,
\end{align*}
where we used the fact that $\tilde g(t_{2s})=0$.
Hence,
\begin{align}\label{first}
\bigl|[t_0,\dots,t_{{2s}+k};\tilde g]\bigr|&\ge\frac{\|g\|}{(t_{2s}-t_1)^{{2s}-1}(t_{2s}-t_0)\prod_{j={2s}+1}^{{2s}+k}(t_i-t_0)}\\
&=\frac{\|g\|}{(t_{2s}-t_1)^{{2s}-1}(t_{2s}-t_0)k!H^k}.\notag
\end{align}
On the other hand,
\[
\bigl|[t_0,\dots,t_{{2s}+k};\tilde g-g]\bigr|=\left|\sum_{i=1}^k\frac{g(t_{{2s}+i})}{\prod_{j=0, j\ne {2s}+ i}^{{2s}+k}(t_{{2s}+i}-t_j)}\right|,
\]
and we have
\[
\prod_{j=0,j\ne {2s}+ i}^{{2s}+k}|t_{{2s}+i}-t_j|
=H^ki!(k-i)!\prod_{j=1}^{{2s}}(t_{{2s}+i}-t_j)\ge H^ki!(k-i)!(t_{{2s}+1}-t_{2s})^{2s}.
\]
Therefore,
\[
\bigl|[t_0,\dots,t_{{2s}+k};\tilde g-g]\bigr|\le\frac{\|g\|H^{-k}}{(t_{{2s}+1}-t_{2s})^{2s}k!}\sum_{i=1}^k\binom ki=\frac{\|g\|H^{-k}}{(t_{2s+1}-t_{2s})^{2s}k!}(2^k-1).
\]
Thus, combining with (\ref{first}), we obtain
\begin{align*}
\frac{\bigl|[t_0,\dots,t_{{2s}+k};\tilde g-g]\bigr|}{\bigl|[t_0,\dots,t_{{2s}+k};\tilde g]\bigr|}&\le\frac{(t_{2s}-t_1)^{{2s}-1}(t_{2s}-t_0)}{(t_{{2s}+1}-t_{2s})^{2s}}(2^k-1)\\
&\le\frac{(2^k-1)(6sh)^{2s}}{(H-6sh))^{2s}}=\frac1{(2^k-1)^{{2s}-1}}
<\frac12.
\end{align*}
Finally, by virtue of (\ref{first}),
\begin{align}\label{dzsh2}
\bigl|[t_0,\dots,t_{{2s}+k}; g]\bigr|&\ge\bigl|[t_0,\dots,t_{{2s}+k};\tilde g]\bigr|-\bigl|[t_0,\dots,t_{{2s}+k};\tilde g-g]\bigr|\\
&\ge\frac12\bigl|[t_0,\dots,t_{{2s}+k};\tilde g]\bigr|\ge\frac{c\|g\|}{h^{{2s}+k}}=\frac{c\|f-f(t_{2s})\|}{h^{{2s}+k}}.\notag
\end{align}
Now, we combine (\ref{dzsh2}) with (\ref{103}) to obtain (\ref{101}).
\end{proof}
\begin{lemma}\label{4.4} Let $h>0$ and $y_{2s}-y_1\le(6s-2)h$.
If $f\in C^{2s-2}\cap\Delta^{(1)}(Y_s)$, then
\begin{equation}\label{1002}
\|f-f(y_1)\|\le ch^{2s-2}\omega_{2}(f^{(2s-2)},h).
\end{equation}
\end{lemma}
\begin{proof}  Note that $f\in\Delta^{(1)}(Y_s)$ implies $\|f-f(y_1)\|=\|f-f(y_1)\|_{[t_1,y_{2s}]}$.

Let $a:=y_1-h$, and set
$
x_i:=y_{i+2},\,0\le i\le 2s-3,\quad x_{2s-2}:=a,\quad x_{2s-1}:=y_1\text{ and } x_{2s}:=y_{2s}.
$
By Lemma \ref{dl} with $l={2s-2}$ and $m=2s$, we obtain
\[
h(y_{2s}-y_1)\bigl|[a,y_1\dots,y_{2s};f]\bigr|\le c\omega_2(f^{({2s-2})},h).
\]
On the other hand, Lemma \ref{m1} yields
\begin{align*}
\bigl|[a,y_1\dots,y_{2s};f]\bigr|&=\frac{\bigl|[y_1\dots,y_{2s};f]\bigr|+\bigl|[a,y_1\dots,y_{2s-1};f]\bigr|}{y_{2s}-a}
\ge\frac{\bigl|[y_1\dots,y_{2s};f]\bigr|}{(6s-1)h}\\
&\ge\frac{\|f-f(y_1)\|_{[y_1,y_{2s}]}}{(y_{2s}-y_1)^{2s-1}(6s-1)h}=c\frac{\|f-f(y_1)\|}{(y_{2s}-y_1)^{2s-1}h}
\end{align*}
Hence,
$$
\|f-f(y_1)\|\le c\frac{(y_{2s}-y_1)^{2s-1}h}{h(y_{2s}-y_1)}\omega_2(f^{({2s-2})},h)\le ch^{2s-2}\omega_2(f^{({2s-2})},h).
$$
This completes the proof.
\end{proof}
Given $t_1<\cdots<t_\nu$, $1\le\nu\le r+1$, denote
\begin{equation}\label{pinu}
\pi_\nu(x)=\prod_{i=1}^\nu(x-t_i).
\end{equation}
\begin{lemma}\label{4.5} Let  $h>0$, $3h\le b-a\le 3(r+1)h$ and $a+h\le t_1<\cdots<t_{r+1}\le b-h$.
If a function $f\in C^r[a.b]$, $r\ge1$, satisfies
$$
f'(x)\pi_{r+1}(x)\ge0,\quad x\in[a,b],
$$
or if $f\in C[a,b]$ changes monotonicity once, at $t_1$, then
\begin{equation}\label{h2}
\|f-f(a)\|_{[a,b]}\le c h^r\omega_2(f^{(r)},h;[a,b]).
\end{equation}
\end{lemma}
\begin{proof} By virtue of Lemma \ref{dl} with $l=r$ and $m=r+2$, setting
\[
x_i:=t_{i+1},\, 0\le i\le r,\quad x_{r+1}:=a\quad\text{and}\quad x_{r+2}:=b,
\]
we readily obtain,
$$
h^2\bigl|[a,t_1,\dots,t_{r+1},b]\bigr|=h^2\bigl|[t_1,\dots,t_{r+1},a,b]\bigr|\le\omega_2(f^{(r)},h;[a,b]).
$$
On the other hand Lemma \ref{m1} implies
$$
(b-a)^{r+2}\bigl|[a,t_1,\dots,t_{r+1},b]\bigr|\ge\|f-f(a)\|_{[a,b]}.
$$
Combining the last two inequalities completes the proof of (\ref{h2}).
\end{proof}
\begin{lemma}\label{4.6} Let $r\ge1$, $h>0$, $3h\le b-a\le 3rh$, and $a+h\le t_1<\cdots<t_r\le b-h$.

If a function $f\in C^r[a,b]$ satisfies
\begin{equation}\label{copos}
f'(x)\pi_{r}(x)\ge0,\quad x\in[a,b],
\end{equation}
then, there is an algebraic polynomial $p=p(\cdot,f;(a,b))$ of degree $\le r+2$, such that $f(a)=p(a)$,
\begin{equation}\label{l138}
\|f-p\|_{[a,b]}\le ch^r\omega_3(f^{(r)},h;[a,b]),
\end{equation}
and
\begin{equation}\label{l139}
p'(x)\pi_{r}(x)\ge0,\quad x\in[a,b],
\end{equation}
\end{lemma}
\begin{proof}
Let $t_0:=a$,  $t_{r+1}:=b$, and denote by $L$ the Lagrange polynomial of degree $<r+2$ interpolating $f'$ at $t_i$, $0\le i\le r+1$. We apply Lemma \ref{dl} with $l=r-1$, $m=r+2$, setting
\[
x_0:=t_{r-1},\quad x_{r-1}:=x,\quad x_{r+1}:=t_0,\quad x_{r+2}:=t_{r+1}\quad\text{and}\quad x_i:=t_i,\,1\le i\le r,\,i\ne r-1,
\]
and get for $x\in(a,b)$, $x\ne t_i$, $1\le i\le r$,
$$
h^2\bigl|[x,a,t_1\dots,t_{r},b;f']\bigr|\min\{|x-a|,|x-t_r|,|x-b|\}\le  c\omega_3(f^{(r)},h;[a,b]).
$$
Hence,
\[
|f'(x)-L(x)|=\bigl|[x,t_0,\dots,t_{r+1};f']\bigr|\prod_{i=0}^{r+1}(x-t_i)|\le ch^{r-1}\omega_3(f^{(r)},h;[a,b]).
\]
Let
\[
p(x):=f(a)+\int_a^xL(t)\,dt.
\]
Then (\ref{l138}) is evident, and since $L(t_i)=f'(t_i)=0,\quad 1\le i\le r$,
$$
p'(x)\pi_r(x)=L(x)\pi_{r}(x)=\frac{f'(a)}{\pi(a)}\pi_{r}^2(x)\frac{x-b}{a-b}+\frac{f'(b)}{\pi(b)}\pi_{r}^2(x)\frac{x-a}{b-a}\ge0,\quad x\in[a,b],
$$
which proves (\ref{l139}).
\end{proof}
\begin{lemma}\label{4.7} Let $r\ge2$, $1\le\nu\le r-1$, $h>0$, $3h\le b-a\le 3\nu h$ and $a+h\le t_1<\cdots<t_\nu\le b-h$.

If a function $f\in C^r[a,b]$ satisfies
\begin{equation}\label{l127}
f'(x)\pi_\nu(x)\ge0,\quad x\in[a,b],
\end{equation}
then there is an algebraic polynomial $p=p(\cdot,f;(a,b))$ of degree $<r+k$, such that $p(a)=f(a)$,
\begin{equation}\label{l128}
\|f-p\|_{[a,b]}\le ch^{r}\omega_k(f^{(r)},h;[a,b]),
\end{equation}
and
\begin{equation}\label{l129}
p'(x)\pi_\nu(x)\ge0,\quad x\in[a,b].
\end{equation}
\end{lemma}
\begin{proof}
Denote $\omega:=\omega_k(f^{(r)},h,[a,b])$. Let $t_0:=a$ and $t_{\nu+i}:=b-h+\frac{ih}{r+k-\nu-2},\, 1\le i\le r+k-\nu-2$.
We apply Lemma \ref{dl} with $l=r-1$, $m=r+k-1$, setting
\[
x_0:=x,\quad x_m=a,\quad\text{and}\quad x_i=t_i, 1\le i\le m-1,
\]
and get
$$
h^k\bigl|[x,t_0,\dots,t_{r+k-2};f']\bigr|=h^k\bigl|[x,t_1,\dots,t_{r+k-2},a;f']\bigr|\le c\omega.
$$
Denote by $L$ the Lagrange polynomial of degree $<r+k-1$, that interpolates $f'$ at $t_i$, $0\le i\le r+k-2$. Then, for $x\in[a,b]$, $x\ne t_i$, $\le i\le r+k-2$,
\[
f'(x)-L(x)=[x,t_1,\dots,t_{r+k-2};f']\prod_{i=0}^{r+k-2}(x-t_i),
\]
whence
$$
|f'(x)-L(x)|\le c_0h^{r-\nu-1}|\pi_\nu(x)|\omega,\quad x\in[a,b].
$$
Now, the desired polynomial may be taken in the form
$$
p(x):=f(a)+\int_a^x\left(L(t)+c_0h^{r-1-\nu}\omega\pi_\nu(t)\right)\,dt.
$$
Evidently, (\ref{l128}) is valid, and by (\ref{l127})
\begin{align*}
p'(x)\pi_\nu(x)&=f'(x)\pi_\nu(x)+(L(x)-f'(x))\pi_\nu(x)+c_0h^{r-1-\nu}\omega\pi_\nu^2(x)\\
&\ge -|(L(x)-f'(x))\pi_\nu(x)|+c_0h^{r-\nu-1}\omega\pi_\nu^2(x)\\&\ge -c_0h^{r-\nu-1}\omega\pi_\nu^2(x)+c_0h^{r-\nu-1}\omega\pi_\nu^2(x)=0,
\end{align*}
which proves (\ref{l129}).
\end{proof}
Finally, we need the next result.
\begin{lemma}\label{mon} Let $r\ge1$, and assume that $f\in C^r[a,b]$ is such that $f'(x)\ge0$, $x\in[a,b]$. Then, there exists a polynomial $P=P(\cdot,f;(a,b))$ of degree $<r+k$ such that
\begin{equation}\label{ls3}
\|f-P\|\le c(b-a)^r\omega_k(f^{(r)},b-a;[a,b]),
\end{equation}
\begin{equation}\label{ls2}
P(a)=f(a),\quad P(b)=f(b),
\end{equation}
and
\begin{equation}\label{ls1}
P'(x)\ge0,\quad x\in[a,b].
\end{equation}
\end{lemma}
\begin{proof} Set $\omega:=\omega_k(f^{(r)},b-a;[a,b])$. Whitney's inequality implies that the best algebraic polynomial $p$ of degree $<r+k-1$, approximating $f'$, satisfies
$$
\|f'-p\|_{[a,b]}\le c_1h^{r-1}\omega.
$$
Hence, the polynomial
$$
\tilde P(x):=f(a)+\int_a^x(p(t)+c_1h^{r-1}\omega)\,dt
$$
of degree $<r+k$, satisfies (\ref{ls3}), (\ref{ls1}) and $\tilde P(a)=f(a)$.

Now, let
$$
P(x):=f(a)+\left(f(b)-f(a)\right)\frac{\tilde P(x)-\tilde P(a)}{\tilde P(b)-\tilde P(a)}.
$$
Evidently, (\ref{ls2}) and (\ref{ls1}) are satisfied, and since
$$
P(x)-f(x)\equiv\left(\tilde P(x)-f(x)\right)+\left(f(b)-\tilde P(b)\right)\frac{\tilde P(x)-\tilde P(a)}{\tilde P(b)-\tilde P(a)},
$$
we get
$$
\|P-f\|_{[a,b]}\le \|\tilde P-f\|_{[a,b]} + |f(b)-\tilde P(b)|\le 2\|\tilde P-f\|_{[a,b]}\le ch^r\omega,
$$
which is  (\ref{ls3}).
\end{proof}

\sect{Proofs of the positive results}
Since $f$ is $2\pi$-periodic, Whitney's inequality implies
$$
\|f-f(0)\|\le c\omega_k(f,2\pi),
$$
that yields, for any $N$ and all $n\le N$, and $f\in\Delta^{(1)}(Y_s)$,
\begin{equation}\label{51}
E_n(f,Y_s)\le E_1(f,Y_s)\le\|f-f(0)\|\le c(N)\omega_m(f,1/n).
\end{equation}
So, in the sequel, we assume that $n>6s$.

Let $x_j:=\frac{j\pi}n$, $I_j:=[x_j,x_{j+1}],\quad j\in\mathbb Z$, $h:=\frac\pi n$, and for each $i\in\mathbb Z$, let $j_i$ be the index, such that $x_{j_i}\le y_i<x_{j_i+1}$. Denote by $O$
the interior of the union
$$
\cup_{i\in\mathbb Z}[x_{j_i-1},x_{j_i+2}],
$$
and let $O_\mu=(x_{\mu^-},x_{\mu^+})$, $\mu\in\mathbb Z$, be the connected components of $O$, enumerated from right to left. Note that $x_{\mu^+}-x_{\mu^-}\le6sh$ for all $\mu\in\mathbb Z$.

Denote $I:=[-\pi,\pi]$. Since $f$ is $2\pi$-periodic and $n>6s$, we may without loss of generality, assume that $\pi\notin O$, so that there is no $O_\mu$ such that $\pm\pi\in O_\mu$. We may also assume, by shifting the indices if needed, that $y_i\in I$, $1\le i\le2s$.

There are two possibilities for the location of the extremum points. Either there is only one component of $O$ in a period $I$, or there are at least two. In the former case, that is, when each $O_\mu$ contains $2s$ points $y_i$, so that Lemma \ref{4.4} and Lemma \ref{4.1}, imply that the polynomial $T_n(x)\equiv f(0)$ provides the validity of the estimates for all appropriate $n$'s, in Theorem \ref{T1}, and Theorems \ref{T2} and \ref{T3}, respectively.

Thus, we are left with the case where there are at least two components of $O$ in $I$.

Under the assumption that $\pi\notin O$, denote by $\Sigma_{n,m}(Y_s)$ the set of continuous piecewise polynomials $S$, on $I$, such that
$$
S\left|_{I_j}\right.=P_j,\quad \text{if}\quad I_j\subset I\setminus O,
$$
and
$$
S\left|_{O_\mu}\right.=p_\mu,\quad O_\mu\subset I,
$$
where $P_j$ and $p_\mu$ are algebraic polynomials of degree $<m$.

In the sequel, by writing $S'(x)$, we implicitly assume that $x\ne x_j$, $x_j\notin O$.

We need the following result \cite{LSS}*{Corollary 4.4}.
\begin{lemma}\label{lss} Let $s\ge1$ and $m\ge 2$. There is a constant $\tilde c=c(s,m)$, such that for $n>\tilde c$, if $f\in C\cap\Delta^{(1)}(Y_s)$, $S\in\Sigma_{n,m}(Y_s)$ and $S'(x)\Pi(x)\ge0$, $x\in I$, where $\Pi(x)$ was defined in $(\ref{Pi})$, then
$$
E_{\tilde cn}(f,Y_s)\le\tilde c(\omega_m(f,1/n)+\|f-S\|_I).
$$
\end{lemma}
\begin{remark} In Lemma $\ref{lss}$, $S\in\Sigma_{n,m}(Y_s)$ is defined on $I$, while in \cite{LSS}*{Corollary 4.4}, $S\in\overline\Sigma_{n,m}(Y_s)$ is defined on $\mathbb R$. But, there is no difficulty in extending $S\in\Sigma_{n,m}(Y_s)$ to an element $S\in\overline\Sigma_{n,m}(Y_s)$.
\end{remark}
We point out that in view of (\ref{51}), we will only have to prove the results for $n>(\tilde c(s,m)+6s)^2=:c_2$. To this end, for all $n>6s$, we construct an appropriate $S$ as follows.

So, we assume that there are at least two components in a period, so that in each component $O_\mu$ there are $1\le\nu\le2s-1\le r+1$ extremal points.

In each interval $I_j\subset I\setminus O$, either $f$ or $-f$ is monotone nondecreasing, so for $r\ge1$, we apply Lemma \ref{mon} to obtain a polynomial $\tilde P_j$, of the same monotonicity, interpolating $f$ at both endpoints of the interval, and such that,
\begin{equation}\label{ls4}
\|f-\tilde P_j\|_{I_j}\le\frac c{n^r}\omega_k(f^{(r)},1/n;I_j).
\end{equation}
For $r=0$ and $k=2$, (\ref{ls4}) follows by taking $\tilde P_j$ to be the linear, interpolating $f$ at both endpoints of $I_j$.

In the proof below we apply the various lemmas of Section 4, where we always take a generic component $O_\mu\subset I$ and $y_{l+i}\in O_\mu$, $1\le i\le\nu<2s$, and we set $y_{l+i}=t_i$, $1\le i\le\nu$. Note that $y_{l+1}$ may be either a maximum of $f$ or of $-f$, and the lemmas are applied accordingly.

By Lemmas \ref{4.5} through \ref{4.7}, for $r=2s-2$, by Lemmas \ref{4.6} and \ref{4.7}, for $r=2s-1$, and by Lemma \ref{4.7}, for $r=2s$, there is a polynomial $\tilde p_\mu$ of degree $<r+2$, $<r+3$ and $<r+k$, $k\ge3$, respectively, comonotone with $f$ on $O_\mu$, such that $\tilde p(x_{\mu^-})=f(x_{\mu^-})$ and, respectively,
\begin{equation}\label{pr}
\|f-\tilde p_\mu\|_{\overline O_\mu}\le\frac c{n^r}\omega_k(f^{(r)},1/n).
\end{equation}
Thus, in all cases, the piecewise polynomial
\[
\tilde S(x):=\begin{cases} \tilde p_\mu(x),&\quad x\in O_\mu\subset I,\\
\tilde P_j(x),&x\in I_j:\,I_j\subset I\setminus O,
\end{cases}
\]
of the respective degrees, is continuous in each interval $(x_{\mu^+},x_{(\mu-1)^-})$, $\mu:O_\mu\subset I$, and by (\ref{ls4}) and (\ref{pr}),  it satisfies, for the respective $k$,
\[
\tilde S'(x)\Pi(x)\ge0,\quad x\in I,\text{ and }\sup_{x\in I}|f(x)-\tilde S(x)|\le\frac{c_3}{n^r}\omega_k(f^{(r)},1/n).
\]
The continuous piecewise polynomial,
$$
S(x):=f(-\pi)+\int_{-\pi}^x\tilde S'(t)\,dt,\quad x\in I,
$$
is in $\Sigma_{n,r+k}(Y_s)$, satisfies $S'(x)\Pi(x)\ge0$, $x\in I$, and since there are at most $2s$ components $O_\mu$, that is, at most $2s$ discontinuities of $\tilde S$, in a period, we conclude that for the respective $k$,
\[
\quad\|f-S\|_I\le2s\frac{c_3}{n^r}\omega_k(f^{(r)},1/n).
\]
Hence, by Lemma \ref{lss}, we obtain
$$
E_{n_1}(f,Y_s)\le \frac c{n_1^r}\omega_k(f^{(r)},1/n_1),\quad n_1>c_2.
$$
This completes the proof.$\qquad\qquad\qquad\qquad\qquad\qquad\qquad\qquad\qquad\qquad\qquad\qquad\quad\qed$

\sect{Counterexamples}

In the proofs below, $G$ is an infinitely differentiable function on $\mathbb R$, satisfying $xG'(x)\ge0$, $x\in\mathbb R$, and
$$
G(x)=\begin{cases}
1,\quad&|x|\ge2,\\
0,\quad&|x|\le1.
\end{cases}
$$
In addition, we assume that $n>c_*$ for some $c_*>0$.

We observe that for $1\le n\le c_*$, $E_n^{(1)}(f,Y_s)\ge E_{[c_*+1]}^{(1)}(f,Y_s)$, and we
apply the case $n=[c_*+1]$, to obtain the results for $1\le n\le c_*$.

We begin with the proof of the negative result involving all $r$, $0\le r<2s-2$, and then we deal with the negative results for the additional, specific $r$'s.

\begin{proof}[Proof of Theorem $\ref{L51}$]
Let $0\le r<2s-2$ and set
$$
\tau(x):=\begin{cases}
2^{r+1}\sin^{r+1}\frac x2,\quad&\text{if\, $r$\, is odd},\\
-\sin^{r+1} x,\quad&\text{if\, $r$\, is even}.
\end{cases}
$$
Note that $|\tau^{(r+1)}(0)|=(r+1))!$. Hence, there is a $0<c_4<1$, such that
\begin{equation}\label{l512}
|\tau^{(r+1)}(x)|\ge\frac12,\quad|x|\le c_4.
\end{equation}
Also, note that for $0<b\le c_4/2$, $\|\tau^{(r)}\|_{[-2b,2b]}\le cb$.

Thus, applying Taylor's expansion, we obtain
\begin{equation}\label{j1j}
\|\tau^{(j)}\|_{[-2b,2b]}\le\frac{(2b)^{r-j}}{(r-j)!}\|\tau^{(r)}\|_{[-2b,2b]}\le cb^{r+1-j},\quad 0\le j\le r.
\end{equation}
Denote by $F$ the $2\pi$-periodic function, such that
$$
F(x):=G\left(\frac xb\right)\tau(x),\quad x\in[-\pi,\pi],
$$
and note that
\begin{equation}\label{l513}
\|F-\tau\|\le(2b)^{r+1}\quad\text{and}\quad\|F^{(r)}-\tau^{(r)}\|\le cb.
\end{equation}
Indeed, by (\ref{j1j}),
\begin{align*}
\|F^{(r)}-\tau^{(r)}\|&=\|F^{(r)}-\tau^{(r)}\|_{[-2b,2b]}\le\|\tau^{(r)}\|_{[-2b,2b]}+\|F^{(r)}\|_{[-2b,2b]}\\
&\le cb+2^r\max_{0\le j<r}\frac1{b^{r-j}}\|\tau^{(j)}\|_{[-2b,2b]}\|G^{(r-j)}\|<cb.
\end{align*}
Hence,
\begin{align}\label{l514}
\omega_2(F^{(r)},t)&\le\omega_2(\tau^{(r)},t)+\omega_2(F^{(r)}-\tau^{(r)},t)\\
&\le t^2\|\tau^{(r+2)}\|+2^2\|F^{(r)}-\tau^{(r)}\|\le c_5t^2+c_6b.\notag
\end{align}
Define the collection $Y_s$ as follows.

If $r$ is even, then we let $y_1=\frac{-\pi}2$, $y_{2s}=\frac\pi2$ and $y_i\in[-b,b]$, $2\le i\le 2s-1$, and we observe that $F(y_1)=1$, is the single maximum in $[-\pi,-b]$, and $F(y_{2s})=-1$, is the single  minimum in $[b,\pi]$.

If $r$ is odd, then we let $y_1=-\pi$ and $y_i\in[-b,b]$, $2\le i\le 2s$, and we observe that $F(y_1)=2^{r+1}$, the single maximum in $[-\pi,\pi)$, outside $[-b,b]$.

Since $F(x)=0$, $x\in[-b,b]$, in both cases we may consider that $F\in\Delta^{(1)}(Y_s)$.

In both cases there are at least $r+1$ points $y_i\in [-b,b]$, and for every trigonometric polynomial $T_n\in\Delta^{(1)}(Y_s)$, of degree $<n$, we have $T'_n(y_i)=0$. It follows that there is a point $\theta\in(-b,b)$, such that $T_n^{(r+1)}(\theta)=0$. Hence, if $n>r+1$,  then (\ref{l512}) and (\ref{l513}), imply
\begin{align*}
\frac12&\le\left|\tau^{(r+1)}(\theta)\right|=\left|\tau^{(r+1)}(\theta)-T_n^{(r+1)}(\theta)\right|\\
&\le n^{r+1}\|\tau-T_n\|\le n^{r+1}\|F-T_n\|+n^{r+1}\|\tau-F\|\\
&\le n^{r+1}\|F-T_n\|+(2b)^{r+1}n^{r+1}<n^{r+1}\|F-T_n\|+\frac14.
\end{align*}
provided $b<1/(8n)$, and we actually take $b=n^{-2}$ for $n>c_*:=\max\{(2/c_4)^{1/2},r+1,8\}$. Then, for $n>c_*$, we obtain
$$
n^r\|F-T_n\|>\frac1{4n}=\frac n{4(c_5+c_6)}\left(\frac{c_5}{n^2}+c_6b\right)\ge cn\omega_2\left(F^{(r)},\frac1n\right).
$$
This completes the proof.
\end{proof}
\begin{proof}[Proof of Theorem $\ref{L52}$]
Let $r=2s-1$ and $0<b<\pi/2$, and denote
$$
\tau(x):=(\cos b-\cos x)\sin^r x.
$$
Note that
\begin{align}\label{om}
&\omega_4(\tau^{(r-1)},t)\le t^4\|\tau^{(r+3)}\|\le c_7t^4,\\
&\tau^{(j)}(0)=0,\quad0\le j<r\quad\text{and}\quad\tau^{(r)}(0)=r!(\cos b-1).\notag
\end{align}
Since $1-\cos b>b^2/4$, there is a $0<\delta<b$, such that
\begin{equation}\label{l152}
-\tau^{(r)}(x)>\frac{b^2}4>0,\quad x\in[-\delta,\delta].
\end{equation}
In addition,
$$
\tau^{(r)}(x)=\sin^2x\,\alpha(x)+(\sin^rx)^{(r)}(\cos b-\cos x),
$$
where $\alpha$ is a trigonometric polynomial such that $\|\alpha\|<c$, and $\|(\sin^r \cdot)^{(r)}\|<c$.

Thus, we conclude that $\|\tau^{(r)}\|_{[-2b,2b]}\le cb^2$.

Again, by Taylor's expansion, we obtain
\begin{equation}\label{j1}
\|\tau^{(j)}\|_{[-2b,2b]}\le\frac{(2b)^{r-j}}{(r-j)!}\|\tau^{(r)}\|_{[-2b,2b]}\le cb^{r+2-j},\quad 0\le j\le r.
\end{equation}
Denote by $f$ the $2\pi$-periodic function, such that
$$
f(x):=G\left(\frac xb\right)\tau(x),\quad x\in[-\pi,\pi],
$$
and note that
\begin{equation}\label{l53}
\|f-\tau\|=\|f-\tau\|_{[-2b,2b]}\le\|\tau\|_{[-2b,2b]}\le(2b)^{r+2}.
\end{equation}
Also, by (\ref{j1}),
\begin{align}\label{l53a}
\|f^{(r-1)}-\tau^{(r-1)}\|&=\|f^{(r-1)}-\tau^{(r-1)}\|_{[-2b,2b]}\le\|\tau^{(r-1)}\|_{[-2b,2b]}+\|f^{(r-1)}\|_{[-2b,2b]}\\
&\le cb^3+2^r\max_{0\le j<r}\frac1{b^{r-1-j}}\|\tau^{(j)}\|_{[-2b,2b]}\|G^{(r-1-j)}\|<cb^3.\nonumber
\end{align}
Define the collection $Y_s$ so that $y_1=-\pi$ and, for $\delta$ of (\ref{l152}), $y_i\in(-\delta,\delta)$, $2\le i\le 2s$.

Recalling that $r$ is odd, we conclude that
$$
F(x):=\int_0^x f(t)\,dt,
$$
is $2\pi$-periodic and is in $\Delta^{(1)}(Y_s)$.

By (\ref{l53}) and (\ref{l53a}),
$$
\|F-\mathcal T\|\le (2b)^{r+3}\quad\text{and}\quad \|F^{(r)}-\mathcal T^{(r)}\|\le cb^3,
$$
where
$$
{\mathcal T}(x):=\int_0^x\tau(t)\,dt.
$$
Hence,
\begin{align}\label{tr}
\omega_4(F^{(r)},t)&\le\omega_4({\mathcal T}^{(r)},t)+\omega_4( F^{(r)}-{\mathcal T}^{(r)},t)\\
&\le t^4\|\mathcal T^{(r+4)}\|+2^4\|F^{(r)}-\mathcal T^{(r)}\|\notag\\
&\le c_7t^4+c_8b^3.\nonumber
\end{align}
Let $T_n\in\Delta^{(1)}(Y_s)$. Then $T_n'(\delta)\ge0$ and $T_n'(y_i)=0$, $2\le i\le2s$.

Thus, there is a $\theta\in(-\delta,\delta)$, such that
$$
T^{(r+1)}_n(\theta)=r![y_2,\dots,y_{2s},\delta;T_n']\ge0.
$$
Therefore, by (\ref{l152}), for $n>r+1$, we have
\begin{align*}
\frac{b^2}4&\le|\tau ^{(r)}(\theta)|\le|\tau ^{(r)}(\theta)- T_n^{(r+1)}(\theta)|=|\mathcal T^{(r+1)}(\theta)-T_n^{(r+1)}(\theta)|\\
&\le n^{r+1}\|\mathcal T-T_n\|\le n^{r+1}\|\mathcal T-F\|+n^{r+1}\|F-T_n\|\\
&\le n^{r+1}(2b)^{r+3}+n^{r+1}\|F-T_n\|\le\frac{b^2}8+n^{r+1}\|F-T_n\|,
\end{align*}
provided $b<\frac1{2^4n}$, and we actually take $b=n^{-4/3}$ for $n>c_*:=\max\{r+1,2^{12}\}$. Then for $n>c_*$, we obtain
by (\ref{tr}),
\[
\|F-T_n\|\ge\frac{\root3\of n}{8n^{r+4}}=\frac{\root3\of n}{8(c_7+c_8)n^r}\left(\frac{c_7}{n^4}+c_8
b^3\right)\ge\frac{c\root3\of n}{n^r}\omega_4(F^{(r)},1/n).
\]
This completes the proof.
\end{proof}
\begin{proof}[Proof of Theorem $\ref{L53}$]
Let $r=2s-2$, and set
$$
\tau(x)=\sin^{r+1}x.
$$
Note that $\tau^{(r+1)}(0)=(r+1)!$, therefore, there is a constant $0<c_9<1$, such that  $\tau^{(r+1)}(x)\ge\frac12$ for $|x|\le c_9$. Observing that $\tau^{(r)}(0)=0$, we conclude that
\begin{equation}\label{l55}
|\tau^{(r)}(x)|\ge\frac{|x|}2,\quad|x|\le c_9.
\end{equation}
For $0<b<c_9$, denote by $f$ the $2\pi$-periodic function, such that
$$
f(x):=G\left(\frac xb\right)\tau(x),\quad x\in[-\pi,\pi],
$$
and we obtain, just as in (\ref{l53}) and (\ref{l53a}),
\begin{equation}\label{l56}
\|f-\tau\|\le(2b)^{r+1},
\end{equation}
and if $r\ge1$, then
\begin{equation}\label{l66}
\|f^{(r-1)}-\tau^{(r-1)}\|\le cb^2.
\end{equation}

Define the collection $Y_s$ so that $y_1=-\pi$ and $y_i\in[b/2,b]$, $2\le i\le 2s$.

Recalling that $r$ is even, we conclude that
$$
F(x):=\int_0^x f(t)\,dt,
$$
is $2\pi$-periodic and is in $\Delta^{(1)}(Y_s)$, and by (\ref{l56}),
$$
\|F-{\mathcal T}\|\le (2b)^{r+2},
$$
where
$$
{\mathcal T}(x):=\int_0^x \tau(t)\,dt.
$$
Thus, by (\ref{l66}),
\begin{align}\label{l57}
\omega_3(F^{(r)},t)&\le\omega_3(\mathcal T^{(r)},t)+\omega_3( F^{(r)}-\mathcal T^{(r)},t)\\
&\le t^3\|\mathcal T^{(r+3)}\|+2^3\|F^{(r)}-\mathcal T^{(r)}\|\notag\\
&\le c_{10}t^3+c_{11}b^2.\nonumber
\end{align}
Since  $T_n'(y_i)=0$, $i=2,\dots, 2s$, for each polynomial $T_n\in\Delta^{(1)}(Y_s)$,
there is a point $\theta\in[b/2,b]$, such that $T^{(r+1)}_n(\theta)=0$.

Hence, by (\ref{l55}), for $n>r+1$, we have
\begin{align*}
\frac b4&\le|\tau ^{(r)}(\theta)|=|\tau ^{(r)}(\theta)- T_n^{(r+1)}(\theta)|=|\mathcal T^{(r+1)}(\theta)-T_n^{(r+1)}(\theta)|\\
&\le n^{r+1}\|\mathcal T-T_n\|\le n^{r+1}\|\mathcal T-F\|+n^{r+1}\|F-T_n\|\\
&\le n^{r+1}(2b)^{r+2}+n^{r+1}\|F-T_n\|\le\frac b8+n^{r+1}\|F-T_n\|,
\end{align*}
provided $b<1/(2^4n)$, and we actually take $b=n^{-3/2}$ for $n>c_*:=\max\{(2/c_9)^{2/3},r+1,2^8\}$. Then for $n>c_*$, we
obtain by (\ref{l57}),
\[
\|F-T_n\|\ge\frac b{8n^{r+1}}=\frac{\sqrt n}{8n^{r+3}}=\frac{\sqrt n}{8(c_{10}+2^3c_{11})n^r}\left(\frac{c_{10}}{n^3}+2^3c_{11}b^2\right)\ge\frac{c\sqrt n}{n^r}\omega_3(F^{(r)},1/n).
\]
This completes the proof.
\end{proof}

\end{document}